\documentclass[11pt]{amsart}
\usepackage[margin=1.2in,marginparsep=0.1in,marginparwidth=1in]{geometry}
\usepackage{amssymb,amsmath,amsthm,amstext,amscd,latexsym,graphics,graphicx,bbm,caption}
\usepackage[dvipsnames,svgnames,table]{xcolor}
\usepackage[plainpages=false,colorlinks=true, pagebackref]{hyperref}
\usepackage{tikz}
\usepackage{mathtools}
\usepackage{tikz-cd}
\usepackage{cleveref}
\usepackage{enumitem}
\usepackage{setspace}
\hypersetup{citecolor=Sepia,linkcolor=blue, urlcolor=blue}

\makeatletter
\def\@tocline#1#2#3#4#5#6#7{\relax
	\ifnum #1>\c@tocdepth 
	\else
	\par \addpenalty\@secpenalty\addvspace{#2}%
	\begingroup \hyphenpenalty\@M
	\@ifempty{#4}{%
		\@tempdima\csname r@tocindent\number#1\endcsname\relax
	}{%
		\@tempdima#4\relax
	}%
	\parindent\z@ \leftskip#3\relax \advance\leftskip\@tempdima\relax
	\rightskip\@pnumwidth plus4em \parfillskip-\@pnumwidth
	#5\leavevmode\hskip-\@tempdima
	\ifcase #1
	\or\or \hskip 2em \or \hskip 2em \else \hskip 3em \fi%
	#6\nobreak\relax
	\dotfill\hbox to\@pnumwidth{\@tocpagenum{#7}}\par
	\nobreak
	\endgroup
	\fi}
\makeatother
\newtheorem{intro-thm}{Theorem}[]
\theoremstyle{plain}
\newtheorem{thm}{Theorem}[section]
\newtheorem{theorem}[thm]{Theorem}

\newtheorem{conjecture}[thm]{Conjecture}
\theoremstyle{definition}

\renewcommand{\P}{{\mathbb P}}

\newcommand{\Br}{\text{Br}}

\input{xy}

\begin{document}
	\title{A Geometric Perspective on Amitsur's Conjecture}
	
    \author[D. C-Ramachandran]{Divyasree C-Ramachandran} \address{Department of Mathematics, IISER Pune, Dr Homi Bhabha Rd, Pashan, Pune, 411008, India}\email{crdivya99@gmail.com}

	\begin{abstract}
Amitsur proved that Severi--Brauer varieties associated with cyclic algebras of the same degree are birationally equivalent whenever their Brauer classes generate the same cyclic subgroup. We give a geometric proof of this result by constructing explicit birational maps in homogeneous coordinates and using Galois descent.

	\end{abstract}
    
\maketitle

\section{Introduction}
Severi--Brauer varieties occupy a central place in the interplay between algebra and geometry. They provide geometric realizations of central simple algebras, and their function fields serve as splitting fields for these algebras. Their birational classification, initiated by Amitsur, reveals deep connections between the geometry of these varieties and the structure of the underlying algebras. 

Let \(F\) be a field. Given a central simple \(F\)-algebra \(A\), consider the Severi--Brauer variety associated with it, denoted by \(\mathrm{SB}(A)\). In 1955, Amitsur proposed the following conjecture in his ground-breaking paper \cite{Amitsur1955}. He was the first to emphasize the importance of birational viewpoint on Severi--Brauer varieties in the study of central simple algebras.
  \begin{conjecture}[Amitsur]\label{Amitsur's conjecture}
    Let \(A\) and \(B\) be central simple algebras of the same degree over a field \(F\).
    Then \(\mathrm{SB}(A) \sim_{\mathrm{bir}} \mathrm{SB}(B)\) if and only if
    the Brauer classes \([A]\) and \([B]\) generate the same cyclic subgroup
    of \(\mathrm{Br}\, F\).
\end{conjecture}
   He proved one of the implications in the same paper.
\begin{theorem}[Amitsur's theorem]
    Let \(X\) be a Severi--Brauer variety defined over a field \(F\). Then the kernel of the restriction map \(r_X: \mathrm{Br}\, F \rightarrow \mathrm{Br}\, F(X)\) is a cyclic subgroup generated by \([X]\) in \(\mathrm{Br}\, F\).
\end{theorem}
As an immediate consequence, if \( \mathrm{SB}(A) \) is birational to \( \mathrm{SB}(B) \), then \([A]\) and \([B]\) generate the same cyclic subgroup of the Brauer group \( \Br \, F \). Conjecture \ref{Amitsur's conjecture} remains open for central simple algebras of prime power index. Many cases have been solved by Tregub \cite{Tregub1991}, Roquette \cite{Roquette1964}, Krashen \cite{Krashen2008}, and \cite{Krashen2003} in the past. More recently, Koll\'ar \cite{Kollar2025} addressed the case when the index of the algebra is not a prime power. Amitsur proved the following theorem in \cite{Amitsur1955}. Roquette \cite{Roquette1964} provided another proof and extended the result to algebras containing a maximal subfield that is Galois with a solvable Galois group. 
\begin{theorem}[Amitsur]\label{Amitsur}
    Let \(A\) be a cyclic algebra over \(F\) with degree \(s\). Then, for any integer \(\ell\) coprime to the period of \(A\), the varieties \(\mathrm{SB}(A)\) and \(\mathrm{SB}(A^{\ell})\) are birational, where \(A^\ell\) denotes a degree \(s\)
    representative of the Brauer class \(\ell[A]\).
\end{theorem}
Our aim is to give a direct geometric construction of a birational map in the cyclic case. Explicit birational constructions for Severi--Brauer varieties associated with cyclic algebras have appeared in earlier work, notably that of Krashen \cite{Krashen2003}; see also Koll\'ar \cite[Section~9]{KollarGeometric2016} for a geometric treatment through norm hypersurfaces. We present a construction in homogeneous coordinates using Galois descent, in which the cocycle is incorporated into explicit constants and the equivariance of the resulting map can be checked directly.

More specifically, we describe the function field of a Severi--Brauer variety associated with a cyclic algebra by base-changing to a cyclic splitting field and equipping the resulting projective space with a Galois semilinear action. We then construct an explicit birational map between the projective spaces endowed with the Galois semilinear actions corresponding to \(A\) and \(A^{\ell}\), and show that this map descends to the associated Severi--Brauer varieties. 

\section{Background}\label{background}

\subsection{Crossed Product Algebras} 
Let \(G\) be a finite group of automorphisms of a field \(K\), and let  \(F = K^G\) be the fixed field of \(G\). Then \(K|F\) is a Galois extension with Galois group \(G\). We now construct an algebra associated to this data. Consider the right vector space \( (K,G) = \bigoplus_{g \in G} u_g K\) over \(K\), where the \(u_g\) are basis elements indexed by the elements of \(G\).  
An arbitrary element of \((K,G)\) has the form \(\sum_{g \in G} u_g a_g\) with \(a_g \in K\). We use the right-action convention for crossed-product algebras, with multiplication defined distributively by the following rules 
\begin{align*}
a u_g & = u_g g(a), \quad a \in K, \, \text{and} \\
u_g u_h & = u_{gh}\,\alpha(g,h),
\end{align*}
where \(\alpha : G \times G \to K^\times\) is a fixed function. Associativity requires that the condition \( (u_g u_h) u_f = u_g (u_h u_f),\) be satisfied for all \(g,h,f \in G \). Using the above relations, this is equivalent to the cocycle condition  
\[
f\cdot\alpha(g,h)\,\alpha(gh,f)
=\alpha(g,hf)\,\alpha(h,f).
\]
where \(f\cdot\alpha(g,h):=f(\alpha(g,h))\). Thus, \(\alpha\in Z^2(G,K^\times)\) is a \(2\)-cocycle, and its cohomology class \([\alpha]\in H^2(G,K^\times)\) determines the crossed product
algebra. We denote this algebra by \((K,G,\alpha)\). It is a central simple algebra over the field \(F = K^G\) with degree \(|G|= [K:F]\).

\subsection{Cyclic algebras}

Cyclic algebras constitute an important subclass of central simple algebras. These are generalizations of quaternion algebras to arbitrary degrees. Let \(K|F\) be a cyclic extension with Galois group
\(G = \{1,\sigma,\ldots,\sigma^{n-1}\}\), where \(n=[K:F]\).
For an element \(\gamma\in F^\times\), we define the \(F\)-algebra
\[
(K,G,\gamma)
=
u_1K\oplus u_\sigma K\oplus\cdots\oplus
u_{\sigma^{n-1}}K,
\]
with multiplication given by
\[
u_{\sigma^i}u_{\sigma^j}
=
\begin{cases}
u_{\sigma^{i+j}}, & i+j<n,\\[6pt]
u_{\sigma^{i+j-n}}\, \gamma, & i+j\geq n.
\end{cases}
\]
This algebra is called a cyclic algebra. It is isomorphic to the crossed product of \(K\) and \(G\) associated with the \(2\)-cocycle
\[
\alpha(\sigma^i,\sigma^j)
=
\begin{cases}
1, & i+j<n,\\
\gamma, & i+j\geq n.
\end{cases}
\]
We refer to \(\alpha\) as the normalized cocycle associated with the cyclic algebra \((K,G,\gamma)\). In particular, \((K,G,\gamma)\) is a central simple \(F\)-algebra. Since \(\gamma\in F^\times\) is central, multiplication by \(\gamma\) on either side gives the same result. Moreover, every crossed product \((K,G,\alpha)\) with \(G\) cyclic is isomorphic to a cyclic algebra for a suitable choice of \(\gamma\in F^\times\). We refer the reader to \cite{Bokhut1991} for further discussion of crossed products and cyclic algebras.

\subsection{Galois descent}
Let \(K|F\) be a Galois extension with Galois group \(G\).
The theory of Galois descent gives a one-to-one correspondence
between isomorphism classes of central simple \(F\)-algebras
of degree \(n\) split by \(K\) and isomorphism classes of
\(F\)-forms of \(\mathbb{P}^{n-1}\) split by \(K\).
This correspondence is described by the cohomological
classification \(H^1\!\left(G,\operatorname{PGL}_n(K)\right)\), using the identification
\(\operatorname{Aut}_K(\mathbb{P}^{n-1}_K)
\simeq \operatorname{PGL}_n(K)\). The varieties arising in this way are called Severi--Brauer varieties. For more details, see
\cite[Theorems 2.4.3 and 5.2.1, Definition~5.1.1]
{GilleSzamuely2017}.

  \section{Amitsur's conjecture for cyclic algebras} \label{SB varieties} 
 \subsection{Function field of Severi--Brauer varieties}\label{Brauer fields}

For a Severi--Brauer variety associated with a cyclic algebra \(A\) over \(F\), Roquette constructs an \(m\)-th Brauer field for every multiple \(m\) of the index of \(A\); see \cite[Section~4]{Roquette1964}. The \(m\)-th Brauer field is the function field of the variety \(SB(B)\), where \(B\) is a central simple \(F\)-algebra of degree \(m\) in the Brauer class of \(A\). We refine this construction in the special case \(m=s\), where \(s\) is the degree of \(A\), by giving a concrete geometric description of the function field of the Severi--Brauer variety \(SB(A)\).

Let \(A=(K,G,\alpha)\) be a cyclic algebra of degree \(s\), where \(G=\langle\sigma\rangle\) is cyclic of order \(s\), and let \(\alpha\in Z^2(G,K^\times)\) be a normalized \(2\)-cocycle defining \(A\). Consider the purely transcendental extension \(K(x)=K(x_0,\ldots,x_{s-1})\) of \(K\). The generator \(\sigma\) induces an \(F\)-automorphism \(\sigma_\alpha\) of \(K(x)\), which is \(K\)-semilinear and is given by
\[
\begin{aligned}
\sigma_\alpha(x_i)&=x_{i+1} \cdot \alpha(\sigma^i,\sigma),
&&0\leq i\leq s-1,\\
\sigma_\alpha(\lambda)&=\sigma(\lambda),
&&\lambda\in K,
\end{aligned}
\]
where the indices are taken modulo \(s\). For \(1\leq i\leq s-1\), set \(y_i=\frac{x_i}{x_0}\). Then \(K(\underline{y})=K(y_1,\ldots,y_{s-1})\) is a purely transcendental extension of \(K\) of transcendence degree \(s-1\).
Since \[\sigma_\alpha(x_0)=x_1\alpha(1,\sigma)=x_1\] for a normalized cocycle, the action of \(\sigma_\alpha\) on the \(y_i\)'s is
given by
\[
\sigma_\alpha(y_i)
=
\frac{x_{i+1}}{x_1}\,
\alpha(\sigma^i,\sigma),
\]
with the indices interpreted modulo \(s\). Thus \(K(\underline y)\) is stable
under the induced semilinear action.

The restriction of \(\sigma_\alpha\) to \(K(\underline y)\)
has order \(s\). Thus the powers
\(1,\sigma_\alpha,\ldots,\sigma_\alpha^{s-1}\) define a
cyclic subgroup \(H=\langle\sigma_\alpha\rangle
\subseteq \operatorname{Aut}_F K(\underline y)\). The fixed field \(K(\underline y)^H\) is the function
field of the variety obtained by descending
\(\mathbb{P}^{s-1}_K\) with the corresponding semilinear
Galois action. This field is referred to as the
\(s\)-th Brauer field in
\cite{Roquette1964}.

After base change to \(K\), the descended variety becomes split and is isomorphic to \(\mathbb P_K^{s-1}\). The above semilinear action on \(K(\underline y)\) specifies its descent datum. The identification of the descended variety with \(\operatorname{SB}(A)\) or \(\operatorname{SB}(A^{\mathrm{op}})\) depends on the convention used for Severi--Brauer varieties and the associated descent datum. Since \([A^{\mathrm{op}}]=-[A]\) in \(\operatorname{Br}(F)\), this distinction does not affect the birationality result in Theorem~\ref{main theorem}, as replacing the exponent \(\ell\) with \(-\ell\) preserves the condition \(\gcd(\ell,s)=1\).
\subsection{Birationality for cyclic division algebras}
Theorem~\ref{Amitsur} is the classical result for arbitrary cyclic
algebras. It is enough, however, to prove the result for cyclic
division algebras. Indeed, let \(A\) be a cyclic algebra of degree
\(ns\) over \(F\), and let \(D\) be the division algebra in its Brauer
class, of degree \(s\). Since \(A\) is cyclic, \(D\) is also cyclic.
If \(\ell\) is coprime to \(\operatorname{per}(A)\), then
\(\gcd(\ell,s)=1\), since the period and index have the same prime
divisors (see \cite[Theorem~2.8.7]{GilleSzamuely2017}). Thus, if \(\mathrm{SB}(D)\sim_{\mathrm{bir}}\mathrm{SB}(D^\ell)\), then, using the stable birational equivalences (see
\cite[Remark~5.4.3]{GilleSzamuely2017})
\[
\mathrm{SB}(A)\sim_{\mathrm{bir}}
\mathrm{SB}(D)\times\mathbb P_F^{s(n-1)}
\quad\text{and}\quad
\mathrm{SB}(A^\ell)\sim_{\mathrm{bir}}
\mathrm{SB}(D^\ell)\times\mathbb P_F^{s(n-1)},
\]
where \(A^\ell\) is a degree \(ns\) representative of the Brauer
class \(\ell[A]\) and \(D^\ell\) is a degree \(s\) representative
of the same class, we obtain \(\mathrm{SB}(A)\sim_{\mathrm{bir}}\mathrm{SB}(A^\ell)\). Our explicit construction below therefore treats the case of cyclic
division algebras, for which the index and degree coincide.

\begin{theorem}\label{main theorem}
    Let \(A\) be a cyclic division algebra over \(F\) of degree \(s\).
    Then, for any integer \(\ell\) coprime to \(s\), we have \(  \mathrm{SB}(A)\sim_{\mathrm{bir}}\mathrm{SB}(A^\ell)\), where \(A^\ell\) is a degree \(s\) representative of \(\ell[A]\) .
\end{theorem}
\begin{proof}
Let \(K\mid F\) be a maximal cyclic splitting field of \(A\),
with Galois group \(G\) generated by \(\sigma\). Then \(A\) is isomorphic to the cyclic algebra \( (K,G,\alpha)\) for some \(\alpha \in Z^2(G,K^\times)\). We may choose \(\alpha\) so that \[
\alpha(\sigma^i,\sigma^j)=
\begin{cases}
1,& i+j<s,\\
\gamma,& i+j\geq s,
\end{cases}
\] for some \(\gamma \in F^{\times}\). For brevity, we write \(\alpha_{i,j}\) instead of \(\alpha(\sigma^i,\sigma^j)\). Let \(\ell\) be an integer such that \(1 \leq \ell<s\) and \((\ell,s)=1\).
 Then \(A^\ell=(K,G,\alpha^\ell)\) is a cyclic algebra of degree \(s\) whose Brauer class
is \(\ell[A]\), where \(\alpha^{\ell}\) is the normalized cocycle corresponding to \(\gamma^{\ell} \in F^{\times}\). The case \(\ell=1\) is trivial, so we
may assume \(1<\ell<s\). 
For each \(0\leq i\leq s-1\), consider the semilinear automorphisms \(\varphi_i,\psi_i\) of \(K[x_0,\ldots,x_{s-1}]\) defined by
\[
\begin{aligned}
\varphi_i(x_j)&=x_{j+i} \cdot \alpha_{i,j},\\
\psi_i(x_j)&=x_{j+i}\cdot \alpha_{i,j}^{\ell},\\
\varphi_i(\lambda)&=\psi_i(\lambda)=\sigma^i(\lambda),
\, \lambda\in K,
\end{aligned}
\]
where the coordinate indices are taken modulo \(s\). These induce semilinear automorphisms of \(\mathbb P_K^{s-1}\), which we again denote by \(\varphi_i,\psi_i\) respectively by abuse of notation.
The induced actions on projective points are given by
\[
\begin{aligned}
\varphi_i([a_0:\cdots:a_{s-1}])
&=
[\sigma^i(a_{i})\alpha_{i,0}:
\cdots:
\sigma^i(a_{i+s-1})\alpha_{i,s-1}],\\
\psi_i([a_0:\cdots:a_{s-1}])
&=
[\sigma^i(a_{i})\alpha_{i,0}^{\ell}:
\cdots:
\sigma^i(a_{i+s-1})\alpha_{i,s-1}^{\ell}],
\end{aligned}
\]
where all coordinate indices are interpreted modulo \(s\). By the cocycle relation, \(\varphi_i\) and \(\psi_i\) are the \(i\)-fold
compositions of \(\varphi_1\) and \(\psi_1\), respectively. Thus they give
the \(K\)-semilinear Galois actions corresponding to \(\sigma^i\) on \(\mathrm{SB}(A\otimes_F K) \cong \P_K^{s-1}\) and \(\mathrm{SB}(A^{\ell}\otimes_F K) \cong \P_K^{s-1}\) respectively, as explained in subsection \ref{Brauer fields}. By Galois descent, it suffices to construct a birational map \(\Theta\) that makes the following diagram commute.

\[
\begin{tikzcd}
\mathbb{P}_K^{s-1} \arrow[r, "\Theta"] \arrow[d, "\varphi_1"'] & \mathbb{P}_K^{s-1} \arrow[d, "\psi_1"] \\
\mathbb{P}_K^{s-1} \arrow[r, "\Theta"] & \mathbb{P}_K^{s-1}
\end{tikzcd}
\]
We define \(\Theta\) by
\[
[a_0: a_1: \dots: a_{s-1}] \mapsto [\beta_0 a_0a_1\dots a_{\ell -1}: \beta_1 a_1 a_2 \dots a_{\ell}: \dots: \beta_{s-1} a_{s-1}a_0 \dots a_{s+\ell-2}]
\] for suitable scalars \(\beta_0,\ldots,\beta_{s-1}\in F^\times\), where the indices are
interpreted modulo \(s\). The map \(\Theta\) is the composition of the maps, \[
[a_0: a_1: \dots: a_{s-1}] \xmapsto{\Theta_1} [a_0a_1 \dots a_{\ell -1}: a_1 a_2 \dots a_{\ell} : \dots : a_{s-1}a_0 \dots a_{s+\ell-2}] 
\] and \[
[a_0: a_1: \dots : a_{s-1}] \xmapsto{\Theta_2} [\beta_0 a_0 : \beta_1 a_1 : \dots : \beta_{s-1} a_{s-1}].
\]
The map \(\Theta_1\) is birational. Indeed, let \(T=\{[a_0:\cdots:a_{s-1}]\in\mathbb P^{s-1}_K:
a_0\cdots a_{s-1}\ne0\}\). This is a dense open subset of \(\P^{s-1}_K\) and \(\Theta_1\) restricts to a morphism \(T\to T\).
Since \(\gcd(\ell,s)=1\), choose integers \(k\geq1\) and \(q\geq0\)
such that \(k\ell=1+qs\). On \(T\), define
\[
\Theta_1'([b_0:\cdots:b_{s-1}])
=
\left[
\prod_{u=0}^{k-1}b_{j+u\ell}
\right]_{j=0}^{s-1},
\]
where the indices are taken modulo \(s\).
A direct computation gives
\[
(\Theta_1'\circ\Theta_1)([a_0:\cdots:a_{s-1}])
= \left[
\prod_{v=0}^{k\ell-1}a_{j+v}
\right]_{j=0}^{s-1} = 
\left[
a_j\left(\prod_{i=0}^{s-1}a_i\right)^q
\right]_{j=0}^{s-1}
=
[a_0:\cdots:a_{s-1}].
\]
The same computation shows that
\(\Theta_1\circ\Theta_1'\) is the identity on \(T\).
Hence \(\Theta_1|_T\) is an automorphism, so \(\Theta_1\) is
birational. The map \(\Theta_2\) is an automorphism, and therefore
\(\Theta=\Theta_2\circ\Theta_1\) is birational.

It remains to choose the scalars
\(\beta_i\) so that the above diagram commutes. A direct computation shows that, for \(0\leq j\leq s-1\), the \(j\)-th coordinate of
\(\Theta\circ\varphi_1([a_0,a_1,\dots,a_{s-1}])\) is
\[
\beta_j
(\sigma(a_{j+1})\alpha_{1,j})
(\sigma(a_{j+2})\alpha_{1,j+1})
\cdots
(\sigma(a_{j+\ell})\alpha_{1,j+\ell-1}),
\]
where all indices are interpreted modulo \(s\).
Similarly, for \(0\leq j\leq s-1\), the \(j\)-th coordinate of
\(\psi_1\circ\Theta([a_0,a_1,\dots,a_{s-1}])\) is
\[
\beta_{j+1}
\sigma(a_{j+1}a_{j+2}\cdots a_{j+\ell})
\alpha_{1,j}^{\ell},
\]
where again the indices are interpreted modulo \(s\). Since the coefficients of \(\Theta\) lie in \(F\), the map
\(\Theta\) commutes with the action of \(\sigma\) on the
coordinates. Thus, the substitution computation above is
precisely the equivariance condition for \(\Theta\) with
respect to the two semilinear actions. 

Since the equality is an equality of projective points, the corresponding coordinate vectors need only be proportional. Thus, for some \(k\in F^\times\), we require 
\[
\frac{\beta_0 \alpha_{1,1} \dots \alpha_{1,\ell-1}}{\beta_1 \alpha_{1,0}^{\ell-1}} = \frac{\beta_1 \alpha_{1,2} \dots \alpha_{1,\ell}}{\beta_2 \alpha_{1,1}^{\ell-1}}
= \dots = \frac{\beta_{s-1} \alpha_{1, 0}\dots \alpha_{1,\ell-2}}{\beta_0 \alpha_{1, s-1}^{\ell-1}} =  k 
\] where all indices are interpreted modulo \(s\). Multiplying these \(s\) equalities, the factors involving
the \(\beta_j\)'s cancel. Since each \(\alpha_{1,j}\)
occurs exactly \(\ell-1\) times in both the numerator and
the denominator, we obtain \(k^s=1\). We now exhibit a solution of the system with \(k=1\), by setting
\[
\beta_0=1,\qquad
\beta_j=\gamma^{\max(0,j+\ell-s)}
\quad (1\leq j\leq s-1).
\]
Since \(\alpha_{1,j}=\gamma\) precisely when \(j=s-1\),
and \(\alpha_{1,j}=1\) otherwise, these choices satisfy
the displayed equalities with \(k=1\). Therefore, the
required scalars
\(\beta_0,\ldots,\beta_{s-1}\in F^\times\) exist, and
the diagram commutes.

Note that the set \(T\) is a dense open subset stable under both
semilinear \(G\)-actions. The map \(\Theta\) restricts to
a morphism \(T\to T\), and its inverse is also
equivariant. By Galois descent, these maps descend to
mutually inverse morphisms between the corresponding
descended open subsets. Hence the descended map is
birational.
\end{proof}


\section*{Acknowledgements}

The author is grateful to her supervisor, Amit Hogadi, for numerous helpful discussions that shaped this work and for his constant encouragement. She also thanks Vivek Sadhu for his suggestions on the manuscript. This work was partially supported by the DST-INSPIRE Fellowship (Reg. No. IF210208). Finally, she acknowledges the editor for the efficient handling of the manuscript and the anonymous referee for a careful and constructive report that significantly improved the exposition.

\end{document}